\documentclass[11pt]{amsart}

\usepackage{graphicx,psfrag}
\usepackage{amssymb,amsmath,amsthm}
\usepackage{enumerate}  
\usepackage{blindtext}
\usepackage{caption}
\usepackage{subcaption}

\usepackage{mathrsfs}
\usepackage{enumitem}

\numberwithin{equation}{section}
\allowdisplaybreaks[4]

\newtheorem{proposition}{Proposition}[section]
\newtheorem{theorem}[proposition]{Theorem}
\newtheorem{lemma}[proposition]{Lemma}
\newtheorem{corollary}[proposition]{Corollary}
\newtheorem{definition}[proposition]{Definition}
\theoremstyle{definition}
\newtheorem{remark}[proposition]{Remark}
\newtheorem{example}{Example}

\begin{document}
\title{Decompositions of signed deficient topological measures}
\author{S. V. Butler, University of California Santa Barbara} 
\address{Department of Mathematics,
University of California Santa Barbara, 
552 University Rd., Isla Vista, CA 93117, USA } 
\email{svbutler@ucsb.edu }  
\date{February 20, 2019}
\maketitle
\begin{abstract}
This paper focuses on various decompositions of topological measures, deficient topological measures, signed topological measures, and signed 
deficient topological measures. These set functions generalize measures and correspond to certain non-linear functionals. 
They may assume  $\infty$ or $-\infty$.
We introduce the concept of a proper signed deficient topological measure
and show that a signed deficient topological measure can be represented as a sum of a signed Radon measure and 
a proper signed deficient topological measure.
We also generalize practically all known results that involve proper deficient topological measures and proper topological measures on
compact spaces to locally compact spaces. We prove that the sum of two proper (deficient) topological measures is a proper (deficient) topological measure.
We give a criterion for a (deficient) topological measure to be proper. 
\end{abstract}

\section{Introduction}

Topological measures (initially called quasi-measures) were introduced by 
J. F. Aarnes in~\cite{Aarnes:TheFirstPaper},~\cite{Aarnes:Pure}, and~\cite{Aarnes:ConstructionPaper}.
These generalizations of measures and corresponding generalizations of linear functionals are connected 
to the problem of linearity of the expectational functional on the algebra of observables in quantum mechanics.  
Applications of topological measures and corresponding non-linear functionals to symplectic topology 
have been studied in numerous papers beginning with ~\cite{EntovPolterovich} (which has been cited over 100 times), and in a monograph ~\cite{PoltRosenBook}. 

Topological measures are defined on open and closed subsets of a topological space, which means that there is no algebraic structure on the  domain.  
They lack subadditivity and other properties typical for measures, and many standard techniques 
often employed for measures and linear functionals are no longer applicable to them. 
Nevertheless, many properties of measures still hold for topological measures. 

The natural generalizations of topological measures are signed topological measures and deficient topological measures.
Signed topological measures of finite norm on a compact space were introduced in~\cite{Grubb:Signed} 
then studied and used in various works, including~\cite{Grubb:SignedqmDimTheory},~\cite{OrjanAlf:CostrPropQlf},~\cite{Svistula:Signed},  
and~\cite{Svistula:Convergence}. 
Deficient topological measures (as real-valued functions on a compact space) were first defined and used by 
A. Rustad and $\emptyset.$ Johansen in~\cite{OrjanAlf:CostrPropQlf} and later independently rediscovered and further developed 
by  M. Svistula in~\cite{Svistula:Signed} and~\cite{Svistula:DTM}.    
In~\cite{Butler:STMLC} we define and study signed deficient topological measures on locally compact spaces. Signed deficient topological measures
generalize topological measures, signed topological measures, and deficient topological measures.  
These set functions may assume $\infty$ or $-\infty$.

In this paper we introduce some new concepts, for example, proper signed deficient topological measures, and prove new results.

We also generalize practically all known results that involve proper deficient topological measures and proper topological measures on 
compact spaces to locally compact spaces. Also, we no longer require that all set functions have a finite "norm". 
We prove that the sum of two proper (deficient) topological measures is a proper (deficient) topological measure.
We give a criterion for a (deficient) topological measure to be proper.

The focus of this paper is various decompositions of topological measures, deficient topological measures, signed topological measures, and signed 
deficient topological measures. 
For example, we show that if  $ \nu$ and $ \mu$ are signed topological measures,  and at least one of them has a finite norm, 
then $\nu - \mu$ is a signed topological measure; and  if $ \mu \le \nu$ then $ \nu - \mu$ is a topological  measure.
In many decompositions we use Radon measures and proper signed deficient topological measures. For example, a signed
deficient topological measure can be represented as the sum of a signed Radon measure and a proper signed deficient topological measure.
For a deficient topological measure $\nu$, the maximal Radon measure majorized by $\nu$ is compact-finite if and only if $ \nu$ is singleton-finite.

In this paper $X$ is a locally compact, connected space. By 
$\mathscr{O}(X)$ we denote the collection of open subsets of   $X $;
by $\mathscr{C}(X)$  the collection of closed subsets of   $X $;
by $\mathscr{K}(X)$  the collection of compact subsets of   $X $.
We denote by $\overline E$ the closure of a set $E$, and  $ \bigsqcup$ stands for a union of disjoint sets.
We say that a signed set function is real-valued if its range is $ \mathbb{R}$. 
When we consider set functions into extended real numbers they are not identically $ \infty$ or $- \infty$. 

\section{Preliminaries}

\begin{definition} \label{MDe2}
Let $X$ be a  topological space and $\nu$ be a set function on a family of subsets of $X$ that contains $\mathscr{O}(X) \cup \mathscr{C}(X)$. 
We say that 
\begin{itemize}
\item
$\nu$ is compact-finite if $ |\nu(K) | < \infty$ for any $ K \in \mathscr{K}(X)$;
\item
$\nu$ is singleton-finite if $ | \nu(\{x\}) | < \infty$ for every $x \in X$;
\item 
$\nu$ is locally finite if every point $x$ has a neighborhood $U_x$ such that $| \nu(U_x)| < \infty$;
\item
$\nu$ is simple if it only assumes  values $0$ and $1$;
\item
a nonnegative set-funciton $ \nu$ is finite if $ \nu(X) < \infty$.
\end{itemize}
\end{definition}

\begin{definition}
A Radon measure  $m$  on $X$ is  a Borel measure that is compact-finite, 
outer regular on all Borel sets, and inner regular on all open sets, i.e.
$ m(E) = \inf \{ m(U): \ \ E \subseteq U, \ U \text{  open  } \} $ for every Borel set $E$, and 
$m(U) = \sup \{  m(K): \ \ K \subseteq U, \  K  \text{  compact  } \}$ for every open set $U$. 
By a signed Radon measure we mean a difference of two Radon measures, at least one of which is finite.
\end{definition}

\noindent
The set of all Radon measures is a positive cone. 

\begin{definition}\label{DTM}
A  deficient topological measure on a locally compact space $X$ is a set function
$\nu:  \mathscr{C}(X)\cup \mathscr{O}(X) \longrightarrow [0, \infty]$ 
which is finitely additive on compact sets, inner compact regular, and 
outer regular, i.e. :
\begin{enumerate}[label=(DTM\arabic*),ref=(DTM\arabic*)]
\item \label{DTM1}
if $C \cap K = \emptyset, \ C,K \in \mathscr{K}(X)$ then $\nu(C \sqcup K) = \nu(C) + \nu(K)$; 
\item \label {DTM2} 
$ \nu(U) = \sup\{ \nu(C) : \ C \subseteq U, \ C \in \mathscr{K}(X) \} $
 for $U\in\mathscr{O}(X)$;
\item \label{DTM3} 
$ \nu(F) = \inf\{ \nu(U) : \ F \subseteq U, \ U \in \mathscr{O}(X) \} $  for  $F \in \mathscr{C}(X)$.
\end{enumerate}
\end{definition} 

\noindent
For a closed set $F$, $ \nu(F) = \infty$ iff $ \nu(U) = \infty$ for every open set $U$ containing $F$.

\begin{remark} \label{DTMsvva}
In~\cite[Section 3]{Butler:DTMLC} it is proved that
a deficient topological measure $ \nu$ is superadditive, i.e. 
if $ \bigsqcup_{t \in T} A_t \subseteq A, $  where $A_t, A \in \mathscr{O}(X) \cup \mathscr{C}(X)$,  
and at most one of the closed sets (if there are any) is not compact, then 
$\nu(A) \ge \sum_{t \in T } \nu(A_t)$. In particular, $\nu$ is monotone, i.e. if $A \subseteq B, \ A,B \in \mathscr{O}(X) \cup \mathscr{C}(X)$ then $ \nu(A) \le \nu(B)$. 
\end{remark}

\begin{definition}\label{TMLC}
A topological measure on $X$ is a set function
$\mu:  \mathscr{C}(X)\cup \mathscr{O}(X)  \longrightarrow [0,\infty]$ satisfying the following conditions:
\begin{enumerate}[label=(TM\arabic*),ref=(TM\arabic*)]
\item \label{TM1} 
if $A,B, A \sqcup B \in \mathscr{K}(X) \cup \mathscr{O}(X) $ then
$
\mu(A\sqcup B)=\mu(A)+\mu(B);
$
\item \label{TM2}  
$
\mu(U)=\sup\{\mu(K):K \in \mathscr{K}(X), \  K \subseteq U\}
$ for $U\in\mathscr{O}(X)$;
\item \label{TM3}
$
\mu(F)=\inf\{\mu(U):U \in \mathscr{O}(X), \ F \subseteq U\}
$ for  $F \in \mathscr{C}(X)$.
\end{enumerate}
\end{definition} 

\begin{definition}\label{SDTM}
A signed deficient topological measure on a locally compact space $X$ is a set function
$\nu:  \mathscr{C}(X)\cup \mathscr{O}(X)  \longrightarrow [ - \infty, \infty ] $  that assumes at most one of $\infty, 
-\infty$ and that is finitely additive on compact sets, inner regular on open sets, and outer regular on closed sets, 
i.e. 
\begin{enumerate}[label=(SDTM\arabic*),ref=(SDTM\arabic*)]
\item \label{SDTM1}
If $C \cap K = \emptyset, \ C,K \in \mathscr{K}(X)$ then $\nu(C \sqcup K) = \nu(C) + \nu(K);$ 
\item \label{SDTM2} 
$\nu(U)=\lim\{\mu(K):K \in \mathscr{K}(X), \  K \subseteq U\}
$ for $U\in\mathscr{O}(X)$;
\item \label{SDTM3}
$\nu(F)=\lim\{\mu(U):U \in \mathscr{O}(X), \ F \subseteq U\}$ for  $F \in \mathscr{C}(X)$.
\end{enumerate}
\end{definition} 

\begin{remark}
In condition \ref{SDTM2} we mean the limit of the net $\nu(C)$ with the index set $\{ C \in \mathscr{K}(X): \ C \subseteq U\}$ ordered
by inclusion. The limit exists and is equal to $\nu(U)$. Condition  \ref{SDTM3} is interpreted in a similar way, 
with the index set  being $\{ U \in \mathscr{O}(X): \ U \supseteq C \}$ ordered by reverse inclusion.
\end{remark}

\begin{remark} \label{byCompacts}
Since we consider set-functions that are not identically $ \infty $ or $ - \infty$, we see that for a signed deficient topological measure $ \nu (\emptyset) = 0$.
If $\nu$ and $\mu$ are signed deficient topological measures that agree on $\mathscr{K}(X)$, then $\nu =\mu$;
if $\nu \le \mu$ on $\mathscr{K}(X)$, then $\nu  \le \mu$.
\end{remark}

\begin{definition} \label{SDTMnorDe}
We define  $\| \nu \| = \sup \{ | \nu(K)|  :  K \in  \mathscr{K}(X) \} $ for a signed deficient topological measure $\nu$.
\end{definition}
\noindent
If $\mu$ is a deficient topological measure then $ \| \mu \| = \mu(X)$.

\begin{definition} \label{STMLC}
A signed topological measure on a locally compact space $X$ is a set function
$\mu: \mathscr{O}(X) \cup \mathscr{C}(X) \longrightarrow [-\infty, \infty]$  that assumes at most one of $\infty, 
-\infty$ and satisfies the following conditions:
\begin{enumerate}[label=(STM\arabic*),ref=(STM\arabic*)]
\item \label{STM1} 
if $A,B, A \sqcup B \in \mathscr{K}(X) \cup \mathscr{O}(X) $ then
$\mu(A\sqcup B)=\mu(A)+\mu(B);$
\item \label{STM2}  
$\mu(U)=\lim\{\mu(K):K \in \mathscr{K}(X), \  K \subseteq U\}
$ for $U\in\mathscr{O}(X)$;
\item \label{STM3}
$\mu(F)=\lim\{\mu(U):U \in \mathscr{O}(X), \ F \subseteq U\}$ for  $F \in \mathscr{C}(X)$.
\end{enumerate}
\end{definition} 

By $DTM(X), TM(X), SDTM(X)$, and $STM(X)$ we denote, respectively,  the collections of all deficient topological measures on $X$, 
all topological measures on $X$,  all signed deficient topological measures on $X$, and all signed topological measures on $X$. 
Let $\mathbf{DTM}(X),  \mathbf{TM}  (X), \mathbf{SDTM} (X)$, and $ \mathbf{STM} (X)$ represent 
subfamilies of  set functions with finite $\| \cdot \|$ from $DTM(X)$, $TM(X)$, $SDTM(X)$, and $STM(X)$, respectively.

\begin{remark} \label{Vloz}
Let $X$ be locally compact. 
We denote by $M(X)$ the collection of all Borel measures on $X$ that are inner regular on open sets and outer regular 
(restricted to $\mathscr{O}(X) \cup \mathscr{C}(X)$). Thus, $M(X)$ includes regular Borel measures and Radon measures. Let $SM(X)$ be the family of signed measures 
that are differences of two measures from $M(X)$, at least one of which is finite. 
In general (see~\cite[Section 3]{Butler:STMLC}),
\begin{align} \label{incluMTD}
 M(X) \subsetneqq  TM(X)  \subsetneqq  DTM(X),
\end{align}
and 
\begin{align} \label{incluSMTD}
 SM(X) \subsetneqq  STM(X)  \subsetneqq  SDTM(X).
\end{align}
\end{remark}

The following easy lemma can be found, for example, in~\cite[Chapter X, $\S$ 50, Theorem A]{Halmos}.
\begin{lemma} \label{HalmEzLe}
If $ C \subseteq U \cup V$, where $C$ is compact, $U, V$ are open, then there exist compact sets 
$K$ and $D$ such that $C = K \cup  D, \ K \subseteq U, \ \ D \subseteq V$.
\end{lemma}

\noindent
The following fact is in~\cite[Chapter XI, 6.2]{Dugundji}:
\begin{lemma} \label{easyLeLC}
Let $K \subseteq U, \ K \in \mathscr{K}(X),  \ U \in \mathscr{O}(X)$ in a locally compact space $X$.
Then there exists a set  $V \in \mathscr{O}(X)$ such that
$ K \subseteq V \subseteq \overline V \subseteq U, \ \overline V \in \mathscr{K}(X). $ 
\end{lemma}

\noindent
The following Definition is from~\cite[Section 2]{Butler:DTMLC}. 

\begin{definition} \label{laplu}
Given signed set function $\lambda: \mathscr{K}(X)  \longrightarrow [-\infty, \infty] $ which assumes at most one of $ \infty, -\infty$
we define two set functions on $\mathscr{O}(X) \cup \mathscr{C}(X)$, 
the positive variation $\lambda^{+}$ and the total variation $| \lambda|$,  
as follows:  \\
for an open subset $U \subseteq X$ let 
\begin{eqnarray} 
\lambda^{+}(U) = \sup \{\lambda(K): \  K \subseteq U,  K \in \mathscr{K}(X) \}; 
\end{eqnarray}
\begin{eqnarray} 
|\lambda| (U) = \sup \{ \sum_{i=1}^n |\lambda(K_i)| : \  \bigsqcup_{i=1}^n K_i \subseteq U, \ K_i \subseteq \mathscr{K}(X),  \, n \in \mathbb{N} \}; 
\end{eqnarray}
and for a closed subset $F \subseteq X$ let 
\begin{eqnarray}
\lambda^{+}(F)  = \inf\{ \lambda^{+} (U) : \ F \subseteq U, \ U \in \mathscr{O}(X)\}; 
\end{eqnarray}
\begin{eqnarray} 
|\lambda| (F)  = \inf\{ |\lambda| (U) : \ F \subseteq U, \ U \in \mathscr{O}(X)\}. 
\end{eqnarray} 
We define the negative variation $\lambda^{-}$ associated with a signed set function $\lambda$  
as a  set function $\lambda^{-} = (- \lambda)^{+}$.  
\end{definition}

\begin{remark} \label{varn}
In~\cite[Sections 2 and 3]{Butler:DTMLC} we show that  $\lambda \le \lambda^{+}$ on $ \mathscr{K}(X)$. $ \lambda^{+}$ and $ | \lambda |$ are deficient topological measures. 
Thus, by Remark \ref{DTMsvva}  they are monotone and superadditive.
Also, $| \lambda \pm \nu| \le |\lambda| +| \nu|$. 
If $ \lambda \le \nu $ on $ \mathscr{K}(X)$  then $ \lambda^{+} \le \nu^+$; if $ \lambda$ is a deficient topological measure then $ \lambda^{+} = \lambda$.
If $\lambda$ is a signed measure, then  $ \lambda^{+}$ and $ | \lambda |$ are the usual positive 
and total variations.
\end{remark}

\section{Proper deficient topological measures}
 


\begin{remark}
If $ \nu $ is a nonnegative monotone set function that is outer regular on each singleton then $ \nu$ is singleton-finite 
iff $ \nu$ is  locally finite. In particular, this holds for deficient topological measures. 
If $ \nu$ is also subadditive (in particular, if $ \nu$ is a measure) then $\nu$ is locally finite iff $ \nu$ is compact-finite.
\end{remark}   

\begin{example} \label{NotsinfE}
Consider $ \mu$ on the Borel subsets of $X$.
For a fixed $x$ let $ \mu(A) = \infty$ if $ x \in A$, and $\mu(A) = 0$ otherwise. 
Then $\mu$ is a regular  Borel measure that is not locally finite. 

Here is an example of a deficient topological measure that is not singleton-finite. Fix $x \in X$, and define 
$ \lambda$ on $ \mathscr{K}(X)$ by:  $ \lambda(K) = \infty $ if $ x \in K$ and $ \lambda(K) = 0$ if $ x \notin K$.  By Remark \ref{varn} 
$ \nu = \lambda^{+}$ is a deficient topological measure, and $ \nu$ is not singleton-finite.
\end{example}

\begin{remark}  \label{sifNOTfncmp}
If $ \nu$ is compact-finite then it is singleton-compact.  For an example of a singleton-finite deficient topological measure 
that is not compact-finite, see the last example in~\cite[Section 6]{Butler:DTMLC}.
\end{remark}
 
\begin{definition} \label{properSDTM}
A signed deficient topological measure $\nu$  is called proper if from $m \le |\nu| $, 
where $m$ is a Radon measure and $| \nu|$ is a total variation of $\nu$, it follows that $m = 0$.
\end{definition}

\begin{remark} \label{properDtm}
If $\nu$ is a deficient topological measure (in particular, a topological measure) 
then $| \nu| = \nu$, and $\nu$ is  a proper deficient topological measure if
from $m \le \nu $, where $m$ is a Radon measure, it follows that $m = 0$. 
A signed deficient topological measure $\nu$  is proper if 
the deficient topological measure $| \nu|$ is proper.   
A signed Radon measure $m$ is proper  iff $|m|$ (which is a Radon measure) is proper, i.e. iff $m=0$.

A finite Radon measure is a regular Borel measure, so our definition of a proper deficient topological measure or a 
proper signed topological measure coincides with definitions in all previous papers.
\end{remark}

\begin{definition} \label{Dnutil}
Let $X$ be locally compact. Suppose a set function $\nu: \mathscr{K}(X) \rightarrow [0, \infty]$ is monotone, finitely additive,
and $ \nu ( \emptyset ) < \infty$. 
Define a set function 
$\widetilde \nu: \mathscr{K}(X) \longrightarrow [0, \infty]$ by:
$$ \widetilde \nu (K) = \inf \{ \sum_{i=1}^n \nu (K_i): \  \  K \subseteq \bigcup_{i=1}^n K_i, \ K_i \in \mathscr{K}(X), \ n \in \mathbb{N} \}. $$
\end{definition}

\begin{lemma} \label{nutil}
The following holds for $\widetilde \nu$:
\begin{enumerate}[label=(p\arabic*),ref=(p\arabic*)]
\item \label{p1}
$\widetilde \nu (\emptyset) = 0$.
\item  \label{p2}
$\widetilde \nu \le \nu$ on $ \mathscr{K}(X)$.
\item  \label{p3}
$\widetilde \nu$ is monotone, i.e. $K \subseteq C$ implies $\widetilde \nu(K) \le \widetilde \nu(C)$. 
\item \label{p4}
$\widetilde \nu$ is subadditive, i.e. $\widetilde \nu(K \cup C) \le \widetilde \nu(K)  + \widetilde \nu(C)$.
\item \label{p5}
If $K \in \mathscr{K}(X)$ then  $\widetilde \nu(K) = \inf \{ \sum_{i=1}^n \nu (C_i): \  \  K = \bigcup_{i=1}^n C_i, \ C_i \in \mathscr{K}(X) \} $.
\item \label{p6}
$\widetilde \nu$ is finitely additive on compact sets, i.e. 
if $C, K $ are disjoint compact sets, then $\widetilde \nu( C \sqcup K) = \widetilde \nu(C) + \widetilde \nu(K)$.
\item \label{p7}
If $\nu$ is compact-finite, then so is $\widetilde \nu$.
\end{enumerate}
\end{lemma}

\begin{proof}
Properties \ref{p1} - \ref{p3} easily follow from the definition of $\widetilde \nu$.
\begin{enumerate}
\item[(p4)]
If $\widetilde \nu(K) = \infty $ or $\widetilde \nu(C) = \infty$ then the subadditivity is obvious. Assume now that $ \widetilde \nu(K)  < \infty$ and 
$\widetilde \nu(C) < \infty$. For $\epsilon > 0$ find compact sets $K_i , \ i=1, \ldots n$ and $ C_j, \ j=1, \ldots, m$ such that
$ K \subseteq \bigcup_{i=1}^n K_i, \ C \subseteq \bigcup_{j=1}^m C_j, \ \sum_{i=1}^n \nu(K_i) - \nu(K) < \epsilon, \ 
\sum_{j=1}^m \nu(C_j)  - \nu(C) < \epsilon$. Since the sets $K_i, C_j$ cover $ C \cup K$ , we have 
$$ \widetilde \nu( K \cup C) \le \sum_{i=1}^n \nu(K_i) + \sum_{j=1}^m \nu(C_j) < \widetilde \nu(K) + \widetilde \nu(C)  + 2 \epsilon.$$
The subadditivity of $\widetilde \nu$ follows.
\item[(p5)]
Let $\alpha = \inf \{ \sum_{i=1}^n \nu (C_i): \  \  K = \bigcup_{i=1}^n C_i, \ C_i \in \mathscr{K}(X) \}$.
Clearly, $ \widetilde \nu(K) \le \alpha$. On the other hand, for any  cover $\{ K_i: i=1, \ldots, n\}$ of $K$ 
we have $K = \cup_{i=1}^n ( K \cap K_i)$, so $ \alpha \le \sum_{i=1}^n \nu(K \cap K _i) \le \sum_{i=1}^n \nu(K_i)$, 
leading to $\alpha \le \widetilde \nu(K)$.  
\item[(p6)]
Suppose $C \sqcup K \subseteq \bigcup_{i=1}^n C_i$ where $C_i \in \mathscr{K}(X)$. 
Note that by monotonicity and finite additivity of $\nu$ on compact sets we have $\nu(C_i) \ge \nu(C_i  \cap C ) + \nu(C_i \cap K)$.
Then 
$$ \sum_{i=1}^n \nu(C_i) \ge \sum_{i=1}^n \nu(C_i \cap C) +  \sum_{i=1}^n \nu(C_i \cap K) \ge \widetilde \nu(C)  + \widetilde \nu(K).$$
It follows that $\widetilde \nu(C \sqcup K)  \ge \widetilde \nu(C)  + \widetilde \nu(K)$. 
By part \ref{p4} the statement follows.
\item[(p7)]
Follows from part \ref{p2}.
\end{enumerate} 
\end{proof}

\noindent
In Definition \ref{Dnutil} as $\nu$  we may use  a deficient topological  measure,  a topological measure or a measure.

\begin{lemma} \label{Lnutilplu}
Let $X$ be locally compact. Suppose a set function 
$\nu: \mathscr{K}(X) \rightarrow [0, \infty]$ is monotone, finitely additive, and $ \nu ( \emptyset) < \infty$. 
Let ${\widetilde \nu}^+$ be the positive variation according to Definition \ref{laplu} of the set function
$\widetilde \nu$ from Definition \ref{Dnutil}. Then 
\begin{enumerate}[label=(\roman*),ref=(\roman*)]
\item \label{t1}
${\widetilde \nu}^+$ is a deficient topological measure.
\item \label{t2}
If $ \nu$ is a deficient topological measure then ${\widetilde \nu}^+ \le \nu$. 
If  $ \nu$ is also compact-finite, then so is ${\widetilde \nu}^+ $; if $ \nu$ is finite, then so is ${\widetilde \nu}^+ $. 
\item \label{t3}
${\widetilde \nu}^+$ is finitely subadditive on $\mathscr{O}(X)$.
\item \label{t4}
$m={\widetilde \nu}^+ $ is (a restriction to $ \mathscr{O}(X) \cup \mathscr{C}(X)$ of) a Borel measure which is inner regular on open sets and outer regular. 
If $\nu$ is compact-finite, 
then ${\widetilde \nu}^+$ is a Radon measure. If $ \nu$ is finite, then  ${\widetilde \nu}^+$ is a regular Borel measure. 
\end{enumerate}
\end{lemma}

\begin{proof}
\begin{enumerate}[label=(\roman*),ref=(\roman*)]
\item
Follows from parts \ref{p1} and \ref{p6} of Lemma \ref{nutil} and Remark \ref{varn}.
\item 
By part \ref{p2} of Lemma \ref{nutil}, $\widetilde \nu \le \nu$ on $\mathscr{K}(X)$.
By Remark \ref{varn} ${\widetilde \nu}^+ \le \nu^{+} = \nu$. The rest is obvious.
\item 
Let $U, V$ be open and $C \subseteq U \cup V$, where $C$ is compact. By Lemma \ref{HalmEzLe},
$C = K \cup D$ where $K, D \in \mathscr{K}(X), \ \ K \subseteq U , \ \ D \subseteq V$. Since $\widetilde \nu$ is subadditive, 
$ \widetilde \nu(C) \le \widetilde \nu(K) + \widetilde \nu(D) \le {\widetilde \nu}^+(U)  + {\widetilde \nu}^+(V). $
Taking supremum over compact sets $C \subseteq U \cup V$ we obtain ${\widetilde \nu}^+(U \cup V )\le {\widetilde \nu}^+(U)  + {\widetilde \nu}^+(V).$
\item 
Follows from  parts \ref{t3}, \ref{t2} and~\cite[Theorem 34, Section 4]{Butler:DTMLC}.
\end{enumerate}
\end{proof}

The next Proposition gives the formula for $m$ on $\nu$-finite open sets when $\nu$ is a deficient topological measure.

\begin{proposition} \label{Unutilplu}
Let $X$ be locally compact, and  $\nu$ be a deficient topological measure on $X$. 
Then for any open set $U$ with $\nu(U)< \infty$  
\begin{align*}
m(U)  = {\widetilde \nu}^+(U) &= \inf \{ \sum_{i=1}^n \nu(A_i) : \ U \subseteq \bigcup_{i=1}^n A_i, \ A_i \in \mathscr{O}(X) \cup \mathscr{K}(X) \}  \\ 
& = \inf \{ \sum_{i=1}^n \nu(K_i) : \ U \subseteq \bigcup_{i=1}^n K_i, \ K_i \in \mathscr{O}(X) \} \\
& = \inf \{ \sum_{i=1}^n \nu(U_i) : \ U \subseteq \bigcup_{i=1}^n U_i, \ U_i \in \mathscr{O}(X) \}  \\
& = \inf \{ \sum_{i=1}^n \nu(U_i) : \ U= \bigcup_{i=1}^n U_i, \ U_i \in \mathscr{O}(X) \} 
\end{align*}
\end{proposition}

\begin{proof}
Let $ \alpha, \beta, \gamma, \delta$ be the right hand side, respectively,  of the first, second, third, and fourth line.
Then $ \alpha \le \beta \le  \gamma =\delta$, where $  \gamma =\delta$ by an argument similar to the one for part \ref{p5} 
of Lemma \ref{nutil}.   
 
We shall show that  ${\widetilde \nu}^+(U) \le \alpha$.
Suppose  $ U \subseteq \bigcup_{i=1}^n A_i, \ A_i \in \mathscr{O}(X) \cup \mathscr{K}(X)$. 
Take any compact $K \subseteq U$. Using Lemma \ref{HalmEzLe} we see that we may 
write $K= \bigcup_{i=1}^n K_i, \ K_i \subseteq A_i$. Then $ \widetilde \nu (K) \le \sum_{i=1}^n \nu(K_i)  \le \sum_{i=1}^n \nu(A_i)$. 
Taking supremum  over $K \subseteq U$ we obtain  ${\widetilde \nu}^+(U) \le \sum_{i=1}^n \nu(A_i)$.
Then  ${\widetilde \nu}^+(U) \le  \alpha$.

Now we shall show that $\gamma \le  {\widetilde \nu}^+(U)$. 
This inequality holds  trivially  if $  {\widetilde \nu}^+(U) = \infty$, so assume that $ {\widetilde \nu}^+(U) < \infty$.
Let $\epsilon>0$. Since $\nu(U) < \infty$, find $K \subseteq U$ such that $ \nu(U) - \nu(K) < \epsilon$, and then
$\nu(U \setminus K) < \nu(U) - \nu(K) < \epsilon$. 
For $K$ choose a cover $K = \bigcup_{i=1}^n K_i$ such that $  \sum_{i=1}^n \nu(K_i) -\widetilde \nu(K) < \epsilon$. 
For each $K_i$ choose $U_i \in \mathscr{O}(X)$ such that $ K_i \subseteq U_i$ and  $  \sum_{i=1}^n ( \nu(U_i) - \nu(K_i)) < \epsilon$.
Then the sets $\{ U\setminus K, U_1, \ldots, U_n\} $ form a cover of $U$, so 
$$ \gamma \le \nu(U \setminus K) + \sum_{i=1}^n \nu(U_i)  \le
\nu(U \setminus K) +  \sum_{i=1}^n \nu(K_i)  + \epsilon \le 
 \widetilde \nu(K)  + 3 \epsilon \le {\widetilde \nu}^+(U) + 3 \epsilon.$$
Thus,  $\gamma \le  {\widetilde \nu}^+(U)$, and the proof is complete.
\end{proof}

\begin{theorem} \label{finN}
Let $\nu$ be a deficient topological measure, and $m ={\widetilde \nu}^+$ be the measure from Lemma \ref{Lnutilplu}.
The following are equivalent:
\begin{enumerate}[label=(\alph{enumi})]
\item \label{sif1}
 $ \nu$ is singleton-finite.
\item \label{sif2}
$\widetilde \nu$ is compact-finite.
\item \label{sif3}
$m$ is compact-finite.
\end{enumerate}
\end{theorem}

\begin{proof}
\ref{sif1} $\Longrightarrow$ \ref{sif2}
Suppose  $ \nu$ is singleton-finite.  Let $K  \in \mathscr{K}(X)$. 
For each $x \in K$ let  $U_x \in \mathscr{O}(X)$ be such that $ x \in U_x $ and $ \nu(U_x) < \infty$. By Lemma \ref{easyLeLC} 
let $V_x \in \mathscr{O}(X)$ be such that 
$ x \in V_x \subseteq \overline{ V_x} \subseteq U_x$ and $ \overline{ V_x} $ is compact. Finitely many $V_x$ cover $K$. So 
$ K \subseteq \bigcup_{i=1}^n V_{x_i} \subseteq \bigcup_{i=1}^n \overline{V_{x_i}}$, and $ \sum_{i=1}^n \nu( \overline{V_{x_i}})  \le  \sum_{i=1}^n \nu(U_x) < \infty$.
Thus, $\widetilde \nu(K) < \ \infty$.  \\
\ref{sif2} $\Longrightarrow$ \ref{sif3}
For an arbitrary compact $C$ choose $W \in \mathscr{O}(X)$ with compact closure such that $ C \subseteq W \subseteq \overline{W} \subseteq X $ and see that 
$m(C) \le m(W)  = \sup\{ \widetilde \nu(D): D \subseteq W, \  D \in \mathscr{K}(X) \} \le \widetilde \nu(\overline{W}) < \infty$.
Thus, the measure $m$ is compact-finite. \\
\ref{sif3}  $\Longrightarrow$  \ref{sif1}
Suppose $m$ is compact-finite. Then for $ m(\{x\}) < \infty$. Since $0 \le \widetilde \nu \le m$ on $ \mathscr{K}(X)$, we have 
$0 \le  \widetilde \nu(\{x\}) < \infty$, and by part \ref{p5} of Lemma \ref{nutil} $\nu(\{x\})  = \widetilde \nu(\{x\}) < \infty$.
\end{proof} 

\begin{corollary} \label{maxRad}
Let $\nu$ be a singleton-finite deficient topological measure, and $m ={\widetilde \nu}^+$ be the measure from Lemma \ref{Lnutilplu}.
Then $m$ is the maximal of Radon measures that is majorised by $\nu$. 
\end{corollary}

\begin{proof}
By Theorem \ref{finN} $m$ is compact-finite.
Suppose $\lambda$ is a Radon measure and  $\lambda \le \nu$. 
Take a compact $K$. 
If $K \subseteq \bigcup_{i=1}^n K_i$, where $K_i \in \mathscr{K}(X)$, then
$\lambda(K)  \le \sum_{i=1}^n \lambda(K_i) \le \sum_{i=1}^n \nu(K_i)$. So  
by Definition \ref{Dnutil} and Remark \ref{varn} $\lambda(K) \le \widetilde \nu(K) \le {\widetilde \nu}^+(K)$. 
By Remark \ref{byCompacts} $\lambda \le {\widetilde \nu}^+  $.  
\end{proof}

\section{Decompositions of deficient topological measures}

\begin{theorem} \label{adnlDTM}
Suppose  $\mu \le \nu$, where $\mu$ is a signed topological measure, and 
$\nu$ is a deficient topological measure.
Suppose one of $\mu, \nu$ is compact-finite and the other one is singleton-finite.

\begin{enumerate}[label=(\roman*),ref=(\roman*)]
\item \label{adnl1}
There exists a singleton-finite deficient topological measure $\lambda$ such that $ \lambda = \nu  - \mu$ on  $\mathscr{K}(X)$.
$ \lambda$ is compact-finite iff $\mu$ and $\nu$ are compact-finite.
\item \label{adnl2}
If $\mu^-(X) < \infty$ (in particular, if $ 0 \le \mu \le \nu$) then $\nu = \mu + \lambda$.
\item \label{adnl3}
If $ \mu, \nu$ have finite norms then $ \lambda$ is finite.
\end{enumerate}
\end{theorem}

\begin{proof}
\begin{enumerate}[label=(\roman*),ref=(\roman*)]
\item 
We shall prove the statement in the case where $\mu$ is compact-finite, $ \nu$ is singleton-finite; 
the other case is similar.
Since $\mu$ is real-valued on $\mathscr{K}(X)$, we may define $\lambda$ as follows:
$$ \lambda(K) = \nu(K)  - \mu(K) \mbox{   on   } \mathscr{K}(X) ; $$  
$$ \lambda(U) = \sup \{ \lambda(K): K \subseteq U, \ \ K \in \mathscr{K}(X) \}  \mbox{   on   } \mathscr{O}(X) ;$$
$$ \lambda (F) = \inf \{ \lambda (U):  F \subseteq U, \ \ U \in \mathscr{O}(X) \}  \mbox{   on   } \mathscr{C}(X).$$

Note that  $\lambda$ is a  singleton-finite non-negative set function. For $K$ compact $\lambda(K) < \infty $ iff $ \nu(K) < \infty$.
So $ \lambda$  is compact-finite iff $ \mu$ and $ \nu$ are.
 
First we shall show that $\lambda$ is monotone on $\mathscr{K}(X)$. 
Suppose $ C \subseteq K, \ C, K \in \mathscr{K}(X)$.  
If $ \nu(K) = \infty$, then $ \lambda(K) = \infty \ge \lambda(C)$. 
Now assume that $ \nu(K) < \infty$.
For $ \epsilon >0$ there exists
$ U \in \mathscr{O}(X)$ such that $ K \subseteq U, \,  \nu(U) - \nu(K) < \epsilon$ and $|\mu(U) - \mu(K)| < \epsilon$. 
Note that  $\nu(U), \nu(U \setminus C), \nu(C), \mu(U), \mu(C) $ are real numbers,
$\mu(U) = \mu(C) + \mu(U \setminus C)$, and $ \nu(U) \ge \nu(U \setminus C) + \nu(C)$.
So we have: 
\begin{align*}
\lambda(K) &= \nu(K) -\mu(K) \ge \nu(U) - \epsilon  - \mu(U) - \epsilon   \\
&\ge \nu(U \setminus C) + \nu(C) - \mu(C) - \mu(U \setminus C) - 2 \epsilon \ge \lambda(C) - 2 \epsilon.  
\end{align*}
Thus,  $\lambda(K) \ge \lambda(C)$. 

Now we shall show that this definition is consistent, i.e. if $K$ is compact, $\lambda(K) = \nu(K) - \mu(K)$, and 
$ \lambda_1 (K) = \inf \{ \lambda (U):  K \subseteq U, \ \ U \in \mathscr{O}(X) \}$, then $\lambda(K) = \lambda_1(K)$.
For any $U \in \mathscr{O}(X)$ with $ K \subseteq U$ we have 
$\lambda(U) \ge \lambda(K) $ so $ \lambda_1(K) \ge \lambda(K)$. 
If $\nu(K) = \infty$, then $\lambda(K) = \infty$, and then $\lambda_1(K) = \infty = \lambda(K)$.
Now assume that $\nu(K)  < \infty$. For $\epsilon > 0$ choose $ U \in \mathscr{O}(X)$ such that $ K \subseteq U, \    
 \nu(U) - \nu(K) < \epsilon $, and $ |\mu(U) - \mu(K) | < \epsilon $. 
It is not hard to see  that $ \lambda(U) < \infty$. 
We may choose a compact $C$ such that $ K \subseteq C \subseteq U, \lambda(U) - \lambda(C) < \epsilon$, and $ | \mu(U) - \mu(C)| < \epsilon$.
Note that $| \nu(C) - \nu(K) | \le  2 \epsilon $ and $| \mu(C) - \mu(K) | \le  2 \epsilon $,
so $| \lambda(C) - \lambda(K)| \le  4 \epsilon$. 
Then 
\begin{align*}
\lambda_1(K) & \le \lambda(U) \le \lambda(C) + \epsilon \le |  \lambda(C) - \lambda(K)| + \lambda(K) + \epsilon \le \lambda(K) + 5 \epsilon
\end{align*} 
It follows that  $\lambda_1(K) \le \lambda(K)$. Thus, $\lambda(K) = \lambda_1(K)$.

Finite additivity of $\lambda$ on disjoint compact sets 
follows from the same property for $\nu$ and $\mu$.
Since outer regularity on closed sets and inner regularity on open sets is built in the definition of $\lambda$, 
we see that $\lambda$ is a deficient topological measure. 
\item
If $ \mu^-(X) < \infty$ then $\mu(A), \nu(A) \in (-\infty, \infty]$ for any $A \in \mathscr{O}(X) \cup \mathscr{C}(X)$.
Let $U$ be open. Checking  all possible cases (namely, $\mu(U), \nu(U) \in \mathbb{R}; \mu(U) \in \mathbb{R}, \nu(U) = \infty; \mu(U), \nu(U) = \infty$)
we see that $\nu(U) = \mu(U) + \lambda(U)$. Then it is not hard to check that $\nu(F) = \mu(F) + \lambda(F)$ for any $ F \in \mathscr{C}(X)$.
Thus, $\nu = \mu + \lambda$.
\item
Easy to see.
\end{enumerate}
\end{proof}

\begin{remark}
When $X$ is compact, $\mu \le \nu$, and both $\mu$ and $\nu$ are topological measures (necessarily finite), 
Wheeler~\cite[Proposition 3.4]{Wheeler},  proved that $\nu - \mu$ is also a topological measure by considering corresponding quasi-linear functionals.
In~\cite[Proposition 6]{Svistula:Signed}
the result was extended to prove that if $X$ is compact, $\mu$ is a signed topological measure with finite norm, and $\nu$ is a finite deficient
topological measure, then $\nu - \mu$  is a finite deficient topological measure. Svistula 
considered $\lambda = \nu - \mu$ on $\mathscr{O}(X) \cup \mathscr{C}(X)$ and proved the additivity and regularity of $ \lambda$.
When $X$ is locally compact, one can not prove Theorem \ref{adnlDTM} using this approach; it does not work 
even for topological measures if they are not  both finite.
\end{remark}
  
\begin{theorem} \label{PropDec}
Let $\nu$ be a singleton-finite deficient topological measure. There is the unique decomposition $\nu = m  + \nu'$,
where $m $ is the maximal Radon measure majorized by $\nu$, and $\nu'$ is a proper singleton-finite deficient topological measure. 
$\nu$ is compact-finite (respectively, finite) iff $\nu'$ and $m$ are. 
$\nu$ is a finite topological measure iff $\nu'$ is a finite proper topological measure and $m$ is a finite Radon measure. 
$\nu$ is a finite measure iff $\nu'$ is a finite proper measure and $m$ is a finite Radon measure.
\end{theorem}

\begin{proof}
As in Corollary \ref{maxRad}, let $m ={\widetilde \nu}^+$ be the maximal Radon measure majorized by $\nu$.
By Theorem \ref{adnlDTM} there exists a singleton-finite deficient topological measure $\nu'$ such that $\nu = m + \nu'$. 
Note that $\nu'$ is proper, otherwise, $m$ would not be the maximal Radon measure majorized by $\nu$, 
which contradicts Corollary \ref{maxRad}.  
Uniqueness of decomposition follows from the maximality of $m$ among Radon measures majorized by $\nu$.

If $ \nu$ is compact-finite (respectively, finite), then by Lemma \ref{Lnutilplu} so is $m$, 
and then by Theorem \ref{adnlDTM} so is $ \nu'$.  

If $\nu$  is a finite topological measure (respectively, a finite measure) then $ m$ is a finite measure, and it is easy to check that
$\nu' = \nu - m$ is a finite topological measure  (respectively, a finite measure).
The rest of the claims are easy.
\end{proof}

\begin{corollary} \label{PropDecCor}
Let $\nu$ be a finite deficient topological measure. There is the unique decomposition $\nu = m  + \nu'$,
where $m $ is a finite Radon measure, and $\nu'$ is a finite proper deficient topological measure. 
\end{corollary}

\begin{proof}
By Theorem \ref{PropDec} $\nu =  m  + \nu'$, where $m$ is a  finite Radon measure 
and $\nu'$ is a finite proper deficient topological measure. 
To show that the decomposition is unique, assume that $\nu =  m  + \nu' = l + \lambda'$, where $m, l$ are finite Radon measures 
and $\nu', \lambda'$ are  finite proper deficient topological measures. 
Since $m$ is the maximal of Radon measures majorized by $\nu$, we see that $m-l \ge0$, and $ 0 \le m-l \le \lambda'$. 
Since $\lambda'$ is proper and  $m-l $ is a Radon measure, we have: $m -l = 0$. Then $ \lambda' = \nu'$ on $\mathscr{K}(X)$, and so $ \lambda' = \nu'$.
\end{proof}

\begin{theorem} \label{epsProper}
\begin{enumerate}[label=(\Roman*),ref=(\Roman*)]
\item 
A singleton-finite deficient topological measure $\nu$ is proper iff 
given a compact $ K$ and $\epsilon >0$, there are finitely many compact sets $K_1, \ldots, K_n$ 
such that $K = \bigcup_{i=1}^n K_i$ and $\sum_{i=1}^n \nu(K_i) < \epsilon $.
\item \label{finPdtmX}
A finite deficient topological measure is proper iff given $\epsilon >0$, there are finitely many open sets $U_1, \ldots, U_n$ 
such that $X = \bigcup_{i=1}^n U_i$ and $\sum_{i=1}^n \nu(U_i) < \epsilon $.
\end{enumerate}
\end{theorem}

\begin{proof}
By  Theorem \ref{PropDec} a deficient topological measure is proper iff $m = {\widetilde \nu}^+ =0$, 
where measure $m$ is as in Corollary \ref{maxRad}.
\begin{enumerate}[label=(\Roman*),ref=(\Roman*)]
\item
Then $m(U) = 0$ for any open set $U$, i.e. $\widetilde \nu (K) =0$ for any compact $K$. 
The statement now follows  from part \ref{p5} of Lemma \ref{nutil}.
\item
Then $m(X) = {\widetilde \nu}^+(X) =0$, and the statement now follows from Proposition \ref{Unutilplu}.
\end{enumerate}
\end{proof}

\begin{remark}
According to Proposition \ref{Unutilplu}, in part \ref{finPdtmX} of Theorem \ref{epsProper} 
one may take sets from $\mathscr{O}(X) \cup \mathscr{K}(X)$ or from $\mathscr{K}(X)$ instead of open sets.
\end{remark}

\begin{remark} \label{PropDefEpsi}
By Remark \ref{properDtm} and Theorem \ref{epsProper},
a signed singleton-finite deficient topological measure $\nu$  is proper iff 
given a compact $K$ and $\epsilon >0$, there are finitely many compact sets $K_1, \ldots, K_n$ 
such that $K = \bigcup_{i=1}^n K_i$ and $\sum_{i=1}^n |\nu| (K_i) < \epsilon $.
A finite signed deficient topological measure $\nu$  is proper iff 
given $\epsilon >0$ there are finitely many open sets $U_1, \ldots, U_n$ 
such that $X = \bigcup_{i=1}^n U_i$ and $\sum_{i=1}^n |\nu| (U_i) < \epsilon $.
\end{remark}

\begin{theorem} \label{sumPropDTM}
The sum of two proper deficient topological measures is a proper deficient topological measure.
\end{theorem}

\begin{proof}
Let $\nu, \ \mu$ be proper singleton-finite deficient topological measures. 
Suppose  $\lambda$ is a Radon measure and $\lambda \le \nu + \mu$.
We may consider $\lambda$ on Borel subsets of $X$. 
Let $ K \in \mathscr{K}(X)$ and $\epsilon>0$. By Theorem \ref{epsProper}, let $\{K_i\}_{i=1}^n$ and 
$\{D_j\}_{j=1}^m$ be such that $K  = \bigcup_{i=1}^n K_i = \bigcup_{j=1}^m D_j, \ \sum_{i=1}^n \nu(K_i) < \epsilon, \ 
\sum_{j=1}^m \mu(D_j) < \epsilon, \ K_i, D_j \in \mathscr{K}(X)$. 
Set $A_1 = K_1, A_2 = K_2 \setminus K_1, \ldots, A_n = K_n \setminus \bigcup_{i=1}^{n-1} K_i$. Similarly, let  
$B_1 = D_1, B_2 = D_2 \setminus D_1, \ldots, B_m = D_m \setminus \bigcup_{j=1}^{m-1} D_j$. Then $ A_i \subseteq K_i, \ B_j \subseteq D_j$, 
the sets $A_i$ are disjoint, the sets $B_i$ are disjoint, and 
$K = \bigsqcup_{i=1}^n A_i = \bigsqcup_{j=1}^m B_j$.
Then $K = \bigsqcup_{i,j} A_i \cap B_j, \ A_i = \bigsqcup_{j=1}^m A_i \cap B_j, \ B_i =  \bigsqcup_{i=1}^n A_i \cap B_j.$ 
Note that $\lambda(A_i \cap B_j) \le \lambda(K_i \cap D_j)  \le \nu(K_i) + \mu(D_j) < \infty$. 
Since $\lambda$ is a Radon measure, it is inner regular on $\sigma$-finite sets (see, for example,~\cite[$\S$7.2]{Folland}). 
Thus, given $ \epsilon>0$, we may choose compact sets $C_{ij} \subseteq A_i \cap B_j$ for $i=1, \ldots, n, \ j =1, \ldots, m$
such that 
$$\sum_{i,j} \lambda(C_{ij} ) > \sum_{i,j} \lambda (A_i \cap B_j) - \epsilon = \lambda(K) - \epsilon.$$
Also, 
\begin{align*}
\sum_{i,j}  \lambda(C_{ij} ) &\le \sum_{i,j} \nu (C_{ij} )  + \sum_{i,j} \mu (C_{ij} )  
=  \sum_{i=1}^n  \nu(\bigsqcup_{j=1}^m (C_{ij} ) + \sum_{j=1}^m  \mu(\bigsqcup_{i=1}^n (C_{ij} ) \\
&\le \sum_{i=1}^n  \nu (K_i) + \sum_{j=1}^m  \mu(D_j) < 2 \epsilon.
\end{align*}
Then $\lambda(K) < 3 \epsilon$.  It follows that $\lambda = 0$ on $ \mathscr{K}(X)$. (This proof is adapted from~\cite[Theorem 3]{Svistula:Signed}).
Therefore, $\lambda = 0$, and $\nu  + \mu$ is a proper deficient topological measure.  
\end{proof}

\begin{remark} \label{propDTMcone}
The sum of two singleton-finite proper deficient topological measures is a singleton-finite proper deficient topological measure.

Clearly, if $\nu$ is a  proper deficient topological measure, then so is $ a \nu$ for any positive $a$. 
Using Theorem \ref{sumPropDTM} we see that
the set of all semifinite proper deficient topological measures is a positive cone.
\end{remark}

\begin{theorem} \label{pravnerM}
Suppose $\mu$ is a  topological measure, $\nu$ is a deficient topological measure, one of $ \mu, \nu$ is singleton-finite, and 
the other is compact-finite.
As in Theorem \ref{PropDec}, let  $\mu = m + \mu', \  \nu = n + \nu_1' $ be decompositions, 
where  $ \mu', \nu'$ are proper singleton-finite deficient topological measures, and  
$m, n$ are maximal Radon measures majorized by $\mu$ and $ \nu$ respectively.
If $ \mu \le \nu$ then $m \le n$.  If  $\mu \le \nu$ and $\nu$ is finite then $m \le n$ and $ \mu' \le \nu'$.
\end{theorem}

\begin{proof}
By Theorem \ref{adnlDTM} write $ \nu = \mu + \lambda$, where $ \lambda$  is a singleton-finite deficient topological measure.
By Theorem \ref{PropDec} we have $ \lambda = l + \lambda'$, where $l$ is the maximal Radon measure majorized by $\lambda$
and $ \lambda'$ is a proper singleton-finite deficient topological measure. Then 
$ \nu = \mu + \lambda = m + \mu' + l + \lambda' = (m + l)  + ( \mu' + \lambda'),$
where
$ \mu' + \lambda'$ is a proper singleton-finite deficient topological measure by Theorem \ref{sumPropDTM}, and 
$m+l$ is a Radon measure majorized by $\nu$. Since $n$ is the maximal Radon measure majorized by $\nu$, we have
$ m+l \le n$, so $ m \le n$.

If $\nu$ is finite then so is $ \mu$, and by Theorem \ref{adnlDTM} $\lambda$ is also finite. 
Then in decomposition  $\nu =(m + l)  + ( \mu' + \lambda')$ as above, $m+l$ is a finite Radon measure, and 
$( \mu' + \lambda')$ is a finite  proper deficient topological measure. 
By Corollary \ref{PropDecCor} $ \nu = n + \nu' $ is the unique decomposition
where $n$ is a finite Radon measure and $\nu'$ is a finite proper deficient topological measure.
Then we must  have $ m+l = n, \  \mu' + \lambda' = \nu'$ and so  $m \le n,  \mu' \le \nu'$.
\end{proof}

\begin{corollary}  \label{CorPrp2a}
Suppose $\mu$ is a  topological measure, $l$ is a finite Radon measure,  and $ \mu \le l$. Then $ \mu$ is a finite Radon measure.
\end{corollary}

\begin{proof}
By Theorem \ref{PropDec} write decomposition $ \mu = m + \mu', \ \ \  l= l + 0$. 
By Theorem \ref{pravnerM} $ \mu' \le 0$, i.e. $ \mu' = 0$. Then $\mu = m$ is a Radon measure, and $m$ is finite 
since $\mu$ is. 
\end{proof} 

\begin{definition} \label{extrDTM}
A deficient topological measure is called extreme if it is an extreme point of the sets of all deficient topological 
measures.
\end{definition}

\begin{lemma} \label{finex}
A finite extreme deficient topological measure is either a finite proper deficient topological measure or 
a finite Radon measure. 
\end{lemma}

\begin{proof}
Let $\nu$ be a (nonzero) finite extreme deficient topological measure. By Theorem \ref{PropDec} we write
$\nu = m  + \nu'$,
where $m $ is a finite Radon measure and $\nu'$ is a finite proper deficient topological measure. We may assume that 
$m(X), \nu'(X) \neq 0$, otherwise  the statement follows. Then $0 < m(X), \nu'(X) \le \nu(X) < \infty$ and, 
adapting an argument from~\cite[Corollary 1]{Svistula:Proper}, we may 
write $ \nu = a \nu_1 + b \nu_2,$ where 
$a = \frac{m(X)}{\nu(X)}, \ b=  \frac{\nu'(X)}{\nu(X)}, \  \nu_1  = \frac{\nu(X)}{m(X)} m, \  \nu_2  = \frac{\nu(X)}{\nu'(X)} \nu'.$
But $ a + b =1$ and $\nu$ is extreme, so either $a=0$ and $\nu = \nu'$,  or $b=0$ and $\nu = m$,
and the statement follows. 
\end{proof}

\begin{lemma}
Let $X$ be locally compact. If $\mu$ is a simple deficient topological measure on $X$ and there exists $x$ such that
$\mu(\{x\}) = 1$ then $\mu$ is a point mass at $x$.
\end{lemma}

\begin{proof}
Let $x \in X$ be such that  $\mu(\{x\}) = 1$. Then by superadditivity of $\mu$ 
we have $\mu(X) \ge \mu(\{x\}) + \mu(X \setminus \{x\})$,
and it follows that $\mu(X \setminus \{x\}) =0$.  
It is easy to see that $\mu$ is  finitely subadditive on compact sets, so by~\cite[Section 4]{Butler:DTMLC} $\mu$ is a 
regular Borel measure on $X$, thus, it is a point mass at $x$.
\end{proof}

\begin{lemma} \label{fmv}
Let $X$ be locally compact.
The following are equivalent for a deficient topological measure $\nu$ that assumes finitely many values:
\begin{enumerate}[label=(\alph{enumi})]
\item \label{lp1}
$\nu$ is proper
\item \label{lp2}
There are open sets $U_1, \ldots, U_n$ such that
$X = \bigcup_{i=1}^n U_i$ and $ \nu(U_i)=0, \ i=1, \ldots, n$.
\item \label{lp3}
$\nu(\{x\}) = 0$ for every $x \in X$.
\end{enumerate}
\end{lemma}

\begin{proof}
By Theorem \ref{epsProper},  \ref{lp1} and \ref{lp2} are equivalent. Clearly,  \ref{lp2} implies \ref{lp3}. 
Now assume \ref{lp3}, and we shall show that it implies \ref{lp2}.  Let $\mathcal{U}$  be a finite open  cover of $X$, 
and suppose $U \in \mathcal{U}$ is such that $\nu(U)  = a > 0$. There exists $ K \in \mathscr{K}(X)$ such that $ K \subseteq U$ and 
$\nu(K)  = a$. Since $U = (U \setminus K ) \sqcup K$ and by superadditivity $\nu(U) \ge \nu(U \setminus K) + \nu(K)$, 
it follows that $\nu(U \setminus K) = 0$.  For each $ x \in K$ there is an open set $V_x$ such that $\nu(V_x) = 0$.
Finitely many of $V_x$ cover $K$, and we call them $V_1, \ldots, V_n$. So 
$ U \subseteq (U \setminus K) \cup V_1 \cup \ldots V_n$, 
and $\nu(U \setminus K) = \nu(V_1) = \ldots =\nu(V_n) = 0$.
Replacing the set $U$ in the cover $\mathcal{U}$  by open sets $U \setminus K, V_1, \ldots, V_n$ and repeating the same 
procedure for each set $U  \in \mathcal{U}$ for which $\nu(U) >0$, we obtain a new finite open cover $\mathcal{W}$ with the property
that $\nu(W) =0$ for each set $ W  \in \mathcal{W}$. Thus, \ref{lp2} is satisfied. 
\end{proof}

\section{Decompositions of signed deficient topological measures}

\begin{theorem} \label{adnlSTM} 
Suppose  $ \nu$ and $ \mu$ are signed topological measures at least one of which has a finite norm. 
Then $  \lambda =\nu - \mu$ is a signed topological measure. If $ \mu \le \nu$ then $ \lambda = \nu - \mu$ is a topological  measure.
\end{theorem}

\begin{proof}
Since a signed topological measure assumes at most one of $ \infty, -\infty$, and at least one of 
$ \| \mu \|,   \| \nu \|$ is finite, it is easy to check that $ \lambda = \nu - \mu$ is a signed topological measure.
If $ \mu \le \nu$ then $ \lambda = \nu - \mu \ge 0$. Note that $ \lambda $ is monotone on $ \mathscr{O}(X) \cup \mathscr{C}(X)$: 
for example, if $ K \subseteq U, K \in \mathscr{K}(X), U \in \mathscr{O}(X)$ then $\lambda(U) = \lambda(K)  + \lambda(U \setminus K) \ge \lambda(K)$. 
If $K \subseteq C, \, K,C \in \mathscr{K}(X)$ then clearly $\lambda(K) \le \lambda(C)$ if $ \lambda(C) = \infty$; and if $\lambda(C) < \infty$, for $ \epsilon>0$ choose
$ U$ such that $U \in \mathscr{O}(X), C \subseteq U, | \lambda(U) - \lambda(C) | < \epsilon$, and then 
$\lambda(C) \ge \lambda(U) - \epsilon = \lambda(K) + \lambda( U \setminus K) - \epsilon \ge \lambda(K) - \epsilon$, so $ \lambda(C) \ge \lambda(K)$. 
Other possible cases are considered similarly. A signed topological measure $ \lambda \ge 0$ and monotone, so 
$\lambda$ is a topological measure.
\end{proof}

\begin{remark} \label{spdtmsum}
If a signed deficient topological measure $\mu$ is proper, then so is $- \mu$. 
If $\mu, \nu$ are proper signed deficient topological measures such that $\mu + \nu$ (or $ \mu - \nu$) is defined, 
then by Remark \ref{varn}
$| \mu \pm \nu| \le |\mu|  + |\nu|$, and $|\mu|  + |\nu|$ is proper by Theorem \ref{sumPropDTM}.  
Then $ \mu + \nu$ (or $ \mu - \nu$) is also a proper signed deficient topological measure.
In particular, if $\mu$ and $\nu$ are proper signed deficient topological measures such that 
at least one of $\| \mu \|, \, \| \nu \| < \infty$
then $ \mu \pm \nu$ is a proper signed deficient topological measure. 
The family of all proper signed topological measures with finite norms is a linear space. 
\end{remark}

\begin{theorem} \label{SDTMprdecG}
Suppose a signed singleton-finite deficient topological measure $ \mu$ is the difference of two deficient topological measures, 
one singleton-finite, and  the other  finite. 
\begin{enumerate}[label=(\roman*),ref=(\roman*)]
\item \label{dep1} 
There is a decomposition $\mu = m + \mu'$, where
$m$ is a signed Radon measure, and $ \mu'$ is a proper singleton-finite signed deficient topological measure.
\item \label{dep2}
If  $ \mu$ is the difference of two finite deficient topological measures then there is a unique decomposition 
$\mu = m + \mu'$, where
$m$ is a signed Radon measure, $ \mu'$ is a proper signed deficient topological measure, 
and  $ \| m \|, \, \| \mu' \| < \infty$.
\item \label{dep3}
If $\mu$ is a signed topological measure with $ \| \mu \| < \infty$ then there is a unique decomposition 
$\mu = m + \mu'$, where
$m$ is a signed Radon measure, $ \mu'$ is a proper signed topological measure, 
and  $ \| m \|, \, \| \mu' \| < \infty$.
\end{enumerate}
\end{theorem}

\begin{proof}
\begin{enumerate}[label=(\roman*),ref=(\roman*)]
\item
Assume $ \mu  = \nu - \lambda$, where $ \nu$ is singleton-finite and $ \lambda$ is finite (in the other case consider $ -\mu$).
By Theorem \ref{PropDec}
write $ \nu = n + \nu', \ \lambda = l + \lambda'$, where $n$ is a Radon measure, $l$ is a finite Radon measure, $ \nu'$ is a singleton-finite
proper deficient topological measure, and $ \lambda'$ is a finite proper deficient topological measure.
Then $ \mu = n+ \nu' -l - \lambda' = (n-l) + (\nu' - \lambda')$, and setting $ m= n-l, \mu' = \nu'- \lambda'$ gives the desired decomposition.   
\item
Assume $ \mu  = \nu - \lambda$, where $ \nu$ and $ \lambda$ are finite. 
In the decomposition $ \mu = m + \mu'$  as in part \ref{dep1} we see that $\| m \| < \infty$ and $ \| \mu' \| < \infty$. 
To show the uniqueness of the decomposition, 
assume that $\mu = m + \mu' = m_1 + \mu_1'$, where $m, m_1$ are  signed Radon measures, 
$ \mu', \mu_1'$ are proper signed deficient topological measures, and 
$ \| m \|, \, \| m_1 \|,\, \|  \mu' \|, \,  \| \mu_1' \| < \infty$. 
Then $m-m_1 = \mu_1' - \mu'$ is a signed Radon measure which is proper by Remark \ref{spdtmsum}. 
By Remark \ref{properDtm} $m-m_1 = \mu_1' - \mu' =0$, and  the decomposition is unique.
\item
If $ \mu $ is a signed topological measure with $ \| \mu \| < \infty$, by 
results in~\cite[Sections 2 and 4]{Butler:STMLC} we may 
write $ \mu = \mu^+ - \mu^-$, where $ \mu^+ $ and $\mu^-$ are finite deficient topological measures. 
Apply part \ref{dep2} and note that $\mu' = \mu - m$  is a signed topological measure.
\end{enumerate}
\end{proof}

\begin{theorem} \label{pravner}
Suppose $\mu, \nu$ are signed topological measures, $ \mu \le \nu$, and $ \| \mu \| < \infty$. 
Let $ \mu = m + \mu'$, where $m$ is a signed Radon measure , $ \mu'$ is a proper signed deficient topological measure, and
$\| m \|, \| \mu' \| < \infty$, be the unique decomposition as in Theorem \ref{SDTMprdecG}.
Assume  that  $  \nu = n + \nu'$, where  $n$ is a signed  Radon measure and 
$\nu'$ is a proper singleton-finite signed deficient topological measure, is the unique decomposition 
(see, for example, Theorem \ref{SDTMprdecG} or Theorem \ref{PropDec}).
\begin{enumerate}[label=(\roman*),ref=(\roman*)]
\item \label{cr1}
If $ \mu \le \nu$ then $m \le n$ and $ \mu' \le \nu'$.
If $ \nu \le \mu$ then $n \le m$ and $ \nu' \le \mu'$.
\item \label{cr2}
If  $ \| \nu \| < \infty$ and $| \mu (A)| \le \nu (A) $ for any $ A \in \mathscr{K}(X)$ (or any $ A \in  \mathscr{O}(X) )$,  
then $-n \le m \le n, \   -\nu' \le \mu' \le \nu'$.
\end{enumerate}
\end{theorem}

\begin{proof}
\begin{enumerate}[label=(\roman*),ref=(\roman*)]
\item
Suppose first  $ \mu \le \nu$. By Theorem \ref{adnlSTM} there is a topological measure $ \lambda = \nu - \mu$, 
so $ \nu = \mu + \lambda$.
Note that $ \lambda$ is singleton-finite.
By Theorem \ref{PropDec}
we have $\lambda = l + \lambda'$, where $l$ is a Radon measure and $\lambda'$ is a proper singleton-finite deficient topological measure. 
Then $ \nu= \mu + \lambda = m + \mu' + l  + \lambda'  = ( m+l) + (\mu' + \lambda')$, where $m+ l$ is a Radon signed measure 
and $\mu'+ \lambda'$ is a proper singleton-finite signed deficient topological measure by Theorem \ref{sumPropDTM}.  
Since $  \nu = n + \nu'$ is the unique decomposition, we have
$n =m + l$ and $\nu' = \mu' + \lambda'$. Then $m \le n$ and $ \mu' \le \nu'$.

Now suppose $ \nu \le \mu$. By Theorem \ref{adnlSTM} there is a topological measure $ \lambda = \mu - \nu$, 
so $ \nu = \mu - \lambda$, and $ \lambda$ is singleton-finite. As above, 
$ \nu= \mu - \lambda = m + \mu' - l  - \lambda'  = ( m-l) + (\mu' - \lambda')$, where $m- l$ is a Radon signed measure 
and $\mu' - \lambda'$ is a proper singleton-finite signed deficient topological measure by Remark \ref{spdtmsum}.  
Since $  \nu = n + \nu'$ is the unique decomposition, we see that $n =m - l$ and $\nu' = \mu' - \lambda'$.  
Then $n \le m$ and $ \nu' \le \mu'$.

\item
Let $ \| \nu \| < \infty$.
If $| \mu (A)| \le \nu (A) $ for any $ A \in \mathscr{K}(X)$ (or any $ A \in  \mathscr{O}(X) )$ then 
$ -\nu \le \mu \le \nu$ and by part \ref{cr1} we have $-n \le m \le n, \   -\nu' \le \mu' \le \nu'$.
\end{enumerate}
\end{proof}

\begin{remark}
The proof of Theorem \ref{pravnerM} is similar to the one for Theorem \ref{pravner}. It uses Theorem \ref{adnlDTM} to get
a deficient topological measure $\lambda$  such that $\nu = \mu + \lambda$; and then uses decompositions of
$\mu, \ \nu, $ and $\lambda$ according to Theorem \ref{PropDec}.
\end{remark} 

\begin{corollary} \label{CorPrp1}
Suppose $\mu$ is a signed singleton-finite topological measure, $l$ is a finite Radon measure,
and $| \mu(A) | \le l(A)$ for any $ A \in \mathscr{K}(X)$ (or any $ A \in  \mathscr{O}(X) )$. Then $\mu$ is a signed Radon measure with finite norm.
\end{corollary}

\begin{proof}
From Definition \ref{SDTMnorDe}  we see that $ \| \mu \| \le l(X) < \infty$.
By Theorem \ref{SDTMprdecG} we may write in the unique way $\mu = m + \mu'$, where
$m$ is a signed Radon measure, $ \mu'$ is a proper signed deficient topological measure, 
and  $ \| m \|, \, \| \mu' \| < \infty$.
For $l$  by Theorem \ref{PropDec} we write $ l= l+ 0$, and this decomposition is unique by Theorem \ref{SDTMprdecG}.
By part \ref{cr2} of Theorem \ref{pravner} $ \mu' = 0$, and $ \mu = m$ is a signed Radon measure with finite norm.
\end{proof}

\begin{corollary} \label{CorPrp2}
Suppose $\mu$ is a signed singleton-finite topological measure. Then $ \mu$ 
is a signed Radon measure with finite norm iff $ | \mu |$ is a finite Radon measure. 
\end{corollary}

\begin{proof}
($\Longleftarrow$) Since $ \mu \le | \mu |$,  by Corollary \ref{CorPrp1} $ \mu$ is a signed Radon measure with finite norm.
($\Longrightarrow$) Clear.
\end{proof} 

\begin{corollary} \label{CorPrp3}
Suppose  $n$ is a Radon measure and $\mu$ is a singleton-finite signed topological measure.
If  $| \mu(K) | \le n(K)$ for any $ K \in \mathscr{K}(X)$ then $ |\mu | $ is a Radon measure. 
\end{corollary} 

\begin{proof}
For any finite disjoint collection of compact sets $\bigsqcup_{i=1}^p K_i \subseteq U$ 
where $ U \in \mathscr{O}(X)$ we have $ \sum_{i=1}^p | \mu(K_i) | \le  \sum_{i=1}^p n(K_i)   = n(\bigsqcup_{i=1}^n K_i ) \le n(U)$, 
so by Definition \ref{laplu} $ | \mu| (U) \le n(U)$. From $ | \mu | \le n$ on $ \mathscr{O}(X) $ we have $ | \mu | \le n$.
By Theorem \ref{adnlDTM} there exists a deficient topological measure $ \lambda$ such that $ | \mu |  + \lambda = n$.
By Theorem \ref{PropDec} write decompositions $ | \mu| = m + \mu', \lambda = l+ \lambda', n=n+0$, 
where $m, l $ are Radon measures, and  $\mu', \lambda'$ are proper singleton-finite deficient topological measures.
Then $ | \mu| + \lambda =  m + \mu' + l+ \lambda' = (m+l) + ( \mu' + \lambda') = n+0$. Thus, $ \mu' + \lambda' = 0$, i.e. $ \mu' = \lambda' =0$,
and $| \mu| = m$ is a Radon measure.
\end{proof} 
 
\begin{example} \label{ExProSDTm}
We will give an example of a proper signed topological measure which is not a (proper) signed measure.
Consider a signed topological measure
$\mu = \nu_1 - \nu_3$ from~\cite[Example 25]{OrjanAlf:CostrPropQlf}. It is easy to represent $X$ as a finite union of open solid sets
each of which contains at most one  point from $P$. By Theorem \ref{epsProper}, $\nu_1$ and $\nu_3$ are proper
topological measures. By Remark \ref{spdtmsum} $\mu$ is a proper signed topological measure. 
Both $\nu_1$ and $ \nu_3$ are finite, inner regular on $ \mathscr{O}(X)$, outer regular on $ \mathscr{O}(X) \cup \mathscr{C}(X)$, so if $\mu$ is a signed measure, 
it is a signed (proper) Radon measure. But this is impossible by Remark \ref{properDtm}, since $\mu \neq 0$. 
Thus, $\mu$ can not be a signed measure. 
\end{example}

\begin{lemma} \label{nuproperVk}
Suppose $\nu$ is a proper singleton-finite signed deficient topological measure. Then $\nu = 0$ iff 
\begin{align} \label{nuravno}
\nu(K \cup C) = \nu(K)  + \nu(C)  -\nu(K \cap C) \mbox{     for any compact sets    } K, C.    
\end{align}
\end{lemma}

\begin{proof}
If $\nu = 0$ then (\ref{nuravno}) obviously holds. Now assume (\ref{nuravno}). Then it follows that for finitely 
many compact sets $K_1, \ldots, K_n$ we have 
$$ \nu( \bigcup_{i=1}^n K_i) = \sum_{i=1}^n \nu(K_i) - \nu(K_1 \cap K_2) -  \nu((K_1 \cup K_2) \cap K_3) -
\ldots - \nu( (\bigcup_{i=1}^{n-1} K_i) \cap K_n).$$
Then $|  \nu( \bigcup_{i=1}^n K_i) | \le  2 \sum_{i=1}^n | \nu| (K_i)$. 
By Remark \ref{PropDefEpsi}, given a compact set $K$ and $\epsilon >0$, there are finitely many compact sets $K_1, \ldots, K_n$ 
such that  $K = \bigcup_{i=1}^n K_i$ and $\sum_{i=1}^n |\nu| (K_i) < \epsilon $. Then 
$ | \nu(K)| = | \nu( (\bigcup_{i=1}^n K_i )) | \le 2 \sum_{i=1}^n | \nu| (K_i) \le 2 \epsilon.$
Thus, $\nu(K) = 0$ for any compact set $K$, so $\nu=0$.
\end{proof}

\begin{theorem} \label{whenSTM}
Suppose a signed singleton-finite deficient topological measure $ \mu$ is the difference of two deficient topological measures, 
one singleton-finite, the other  finite. 
(In particular, $\mu$ is a signed topological measure with $\| \mu \| < \infty$.)
Then $\mu$ is a signed Radon measure iff equality (\ref{nuravno}) holds for $\mu$.
\end{theorem}

\begin{proof}
By Theorem \ref{SDTMprdecG} there is a decomposition $\mu = m  + \mu'$,
where $m $ is a signed Radon measure and $\mu'$ is a proper singleton-finite signed deficient topological measure. 
Since (\ref{nuravno}) is satisfied for signed measures, we see that 
(\ref{nuravno}) holds for $\mu$ iff (\ref{nuravno}) holds for $\mu'$, which by Lemma  \ref{nuproperVk}  happens iff 
$\mu' = 0$, i.e.  iff $\mu = m$ is a signed Radon measure.
\end{proof}  

\begin{remark}
If $X$ is compact and $ \mu$ is a signed topological measure with finite norm, then 
$\mu$ is a regular Borel signed measure on $X$ if $ \mu(U) + \mu(V) = \mu(U \cup V) + \mu(U \cap V)$
for $U, V \in \mathscr{O}(X)$. See~\cite[Theorem1]{Grubb:SignedqmDimTheory}.
\end{remark}
  

\begin{remark}
When $X$ is compact and $\nu$ is a finite topological measure, Theorem \ref{PropDec} was first proved by D. Grubb and 
T. LaBerge, see~\cite[Theorem 2.1]{GrubbLaberge}. When $X$ is compact and $\nu$ is a finite deficient topological measure
Theorem \ref{PropDec} was first proved by M. Svistula, see~\cite[Theorem 1]{Svistula:Signed}.
Part \ref{finPdtmX} of Theorem \ref{epsProper}, Theorem \ref{sumPropDTM}, and  Theorem \ref{pravnerM}
generalize~\cite[Theorems 2, 3, and Proposition 7]{Svistula:Signed}, which in turn generalize corresponding results for
topological measures in~\cite{Svistula:Proper}. Lemma \ref{finex}, Lemma \ref{fmv}, and Corollary \ref{CorPrp2a} generalize 
corresponding results from~\cite[pp. 675, 677]{Svistula:Proper}.
Part \ref{dep3} of Theorem \ref{SDTMprdecG}, Theorem \ref{pravner}, Corollary \ref{CorPrp1}, Corollary \ref{CorPrp2}, 
Lemma \ref{nuproperVk}, and Theorem \ref{whenSTM} 
generalize~\cite[Theorem 6, Proposition 9, Corollaries 1, 2, 3, and Propositions 10, 11]{Svistula:Signed}. 
In all  previous results the space was compact and all topological measures and deficient topological measures were finite.  
In our new results the space is locally compact and we no longer require set functions to be finite. 
\end{remark}

\begin{remark}
Topological measures and deficient  topological measures correspond to certain non-linear functionals, see~\cite{Butler:ReprDTM}. 
Thus, in a manner similar to one in~\cite{Svistula:Proper} and~\cite{Svistula:Signed},
 the results involving proper deficient topological measures can be transferred to corresponding non-linear functionals. 
 \end{remark}
 
{\bf{Acknowledgments}}:
The author would like to thank the Department of Mathematics at the University of California Santa Barbara for its supportive environment.

\end{document}